\documentclass[11pt]{article}
\usepackage{amsmath, amssymb, amsthm, mathrsfs}
\usepackage{geometry}
\usepackage{enumitem}
\usepackage{hyperref}
\hypersetup{colorlinks=true, linkcolor=blue, citecolor=blue, urlcolor=blue}

\newtheorem{theorem}{Theorem}[section]
\newtheorem{proposition}[theorem]{Proposition}

\newtheorem{lemma}[theorem]{Lemma}
\theoremstyle{definition}
\newtheorem{definition}[theorem]{Definition}
\newtheorem{remark}[theorem]{Remark}

\title{Spherical 2-Designs from Finite Group Orbits}
\author{Kuan-Cheng Chien and Ming-Hsuan Kang}
\date{\today}

\begin{document}

\maketitle

\begin{abstract}
We classify all spherical 2-designs that arise as orbits of finite group actions on real inner product spaces. Although it is well known that such designs can occur in representations without trivial components, we give a complete characterization of the orbits that satisfy the second-moment condition. In particular, we show that these orbits correspond to projections of compact group orbits within the regular representation, and we provide an explicit classification via isotypic decomposition and moment conditions. This approach unifies geometric and representation-theoretic viewpoints on highly symmetric point configurations.
\end{abstract}

\section{Introduction}

Spherical designs are finite subsets of the unit sphere that exhibit high degrees of symmetry and uniform distribution. A finite set \( X \subset S^{n-1} \subset \mathbb{R}^n \) is called a \emph{spherical \( t \)-design} if it exactly integrates all polynomials \( f \in P_{\le t}(\mathbb{R}^n) \) of degree at most \( t \):
\[
\frac{1}{|X|} \sum_{x \in X} f(x) = \int_{S^{n-1}} f(x)\, d\mu(x),
\quad 
\text{where} \quad d\mu \text{ is the normalized rotationally invariant measure.}
\]
First introduced by Delsarte, Goethals, and Seidel~\cite{delsarte1977spherical}, spherical designs have become a fundamental tool in approximation theory~\cite{womersley2018efficient}, energy minimization~\cite{cohn2007universally}, and the study of highly symmetric configurations~\cite{bannai2009survey}. Classical constructions and generalizations appear in the pioneering work of Bannai~\cite{bannai1984spherical,bannai1979tight,bannai1980tight} as well as in Bajnok~\cite{bajnok1991construction,bajnok1992construction}, Hoggar~\cite{hoggar1990t}, Sobolev~\cite{sobolev1962cubature}, and Goethals–Seidel~\cite{goethals1981cubature}. A clear overview connecting these ideas to finite group representations is provided by Bannai’s survey with Bannai~\cite{bannai2009survey} and by de la Harpe and Pache~\cite{delaharpe2005spherical}.

Among the many families of spherical designs, those that arise as orbits of finite group actions are particularly tractable and geometrically natural. Given a finite group \( G \subset \mathrm{O}(n) \) acting on \( \mathbb{R}^n \) and a unit vector \( v \in S^{n-1} \), the orbit \( Gv \) forms a finite, symmetric subset of the sphere. This motivates the central question of this paper:

\begin{center}
\emph{Which group orbits \( Gv \subset S^{n-1} \) form spherical 2-designs, and how can they be classified?}
\end{center}

While spherical \( t \)-design orbits with \( t \ge 4 \) are extremely rare, an important structural fact is that if a finite group \( G \) admits a 4-design orbit, then \emph{all} its orbits must be spherical 2-designs — as shown by Bannai~\cite{bannai1984spherical}. This condition is equivalent to the standard representation of \( G \) being real irreducible, as explained by de la Harpe and Pache~\cite{delaharpe2005spherical}

Consequently, only the cases \( t = 2 \) and \( t = 3 \) allow much more flexibility. In particular, spherical 2-designs strike a compelling balance between geometric regularity and algebraic tractability, making them the natural focus of this work.

A classical result of Bannai~\cite{bannai1984spherical} provides a necessary condition: for an orbit \( \Omega = Gv \subset S^{n-1} \) to be a spherical \( t \)-design, the associated permutation representation must contain all harmonic representations of degree up to \( \lfloor t/2 \rfloor \). These harmonic representations arise from the action of \( G \subset \mathrm{O}(n) \) on the space of homogeneous harmonic polynomials (see also de la Harpe–Pache~\cite{delaharpe2005spherical} for accessible summaries).

Since any permutation representation embeds into the regular representation \( \mathbb{R}[G] \)—and coincides with it for regular orbits—this condition becomes a concrete algebraic constraint: the regular representation must contain all harmonic components of degree up to \( \lfloor t/2 \rfloor \).

For \( t = 2 \), Bannai’s condition reduces to requiring that the standard representation of \( G \) lies entirely within the orthogonal complement of the trivial component in \( \mathbb{R}[G] \), denoted \( \mathbb{R}[G]_0 \). This subspace contains all nontrivial isotypic components and provides the natural setting for our analysis.

Each isotypic component of \( \mathbb{R}[G]_0 \) corresponds to a real irreducible representation of \( G \) and can be identified with a matrix space \( \mathrm{Mat}_n(D) \), where \( D \in \{ \mathbb{R}, \mathbb{C}, \mathbb{H} \} \) reflects its Frobenius–Schur type. Within these matrix models, the second-moment conditions for spherical 2-designs translate into explicit algebraic constraints involving Frobenius norms and \( D \)-Hermitian inner products.

However, this representation-theoretic condition is not sufficient on its own: only certain orbits actually satisfy the geometric second-moment condition required for a true spherical 2-design.

\medskip

The main goal of this paper is to classify all group orbits \( Gv \subset \mathbb{R}[G]_0 \) that form spherical 2-designs. Our approach relies on two key ingredients:
\begin{itemize}
  \item A decomposition principle showing how global 2-designs decompose into dimension-weighted combinations of local designs in orthogonal \( G \)-invariant subspaces;
  \item A characterization of 2-designs within each isotypic component via explicit moment conditions on the associated matrix data.
\end{itemize}

Together, these yield the following classification result:

\begin{theorem}[Classification of Spherical 2-Design Orbits]\label{thm:main2}
Let \( V \subset \mathbb{R}[G]_0 \) be a real representation of a finite group \( G \). Then a \( G \)-orbit \( Gv \subset S(V) \) forms a spherical 2-design if and only if:
\begin{itemize}
  \item[(i)] For each isotypic component \( V_i \subset V \), there exists a matrix \( M_i \in \mathrm{Mat}_{n_i \times m_i}(D_i) \) satisfying
  \[
  \overline{M}_i^\top M_i = \frac{1}{m_i} I_{m_i}
  \quad \text{and} \quad
  \pi_{V_i}(v) = \sqrt{ \frac{\dim V_i}{\dim V} } \, M_i.
  \]
  \item[(ii)] The full vector \( v \) is the direct sum of the scaled components:
  \[
  v = \bigoplus_i \sqrt{ \frac{\dim V_i}{\dim V} } \, M_i.
  \]
\end{itemize}

Equivalently, every spherical 2-design orbit in \( V \subset \mathbb{R}[G]_0 \) arises as the image of a group orbit
\[
G \Big( \bigoplus_i \sqrt{ \frac{n_i d_i}{\dim V} } \cdot \pi_{n_i, m_i}^{D_i}(U_i) \Big),
\quad \text{where } U_i \in \mathrm{O}_{D_i}(n_i).
\]
\end{theorem}

A key tool in the proof is a unified \( D \)-valued formalism for Schur orthogonality, which recasts real orthogonality conditions in terms of \( D \)-valued identities. This framework is essential for verifying the moment conditions across different Frobenius–Schur types.

\medskip

The remainder of the paper is organized as follows:

\begin{itemize}
  \item Section~2 develops the \( D \)-valued formalism and proves the \( D \)-equivalence theorem.
  \item Section~3 formulates the second-moment condition for group orbits and builds toward a full classification. Section~3.2 establishes the bijection between global designs and tuples of local designs in orthogonal subspaces. Section~3.3 characterizes 2-designs within single isotypic components using explicit matrix conditions. Section~3.4 combines these results to prove Theorem~\ref{thm:main2}.
\end{itemize}

\section{Preliminaries on Real Schur Orthogonality and Its $D$-Valued Formalism}

This section develops a unified $D$-valued framework for Schur orthogonality in real representations, treating all Frobenius–Schur types \( D \in \{ \mathbb{R}, \mathbb{C}, \mathbb{H} \} \) in parallel. While the real form of Schur orthogonality is generally assumed known, we are not aware of an explicit formulation in the literature that applies uniformly across all types. For completeness, we present a precise statement and sketch its proof.

We equip each irreducible real representation \( V \cong D^n \) with a $G$-invariant $D$-valued inner product and analyze matrix coefficients via real bases derived from $D$-orthonormal systems. This approach allows for a coherent treatment of orthogonality relations and translates real inner product identities into uniform $D$-valued formulations.

\subsection{Matrix Coefficients from a $D$-Orthonormal Basis}

Let \( (\rho, V) \) be an irreducible real representation of a finite group \( G \), with \( D := \mathrm{End}_G(V)^{\mathrm{op}} \in \{ \mathbb{R}, \mathbb{C}, \mathbb{H} \} \). View \( V \cong D^n \) as a right \( D \)-module, equipped with a \( G \)-invariant $D$-valued inner product.

Fix a right $D$-orthonormal basis \( \{e_1, \dots, e_n\} \). Let \( \mathcal{I}_D \) denote the standard real basis of \( D \), and define the real orthonormal basis
\[
\mathcal{B} := \{ e_j^{(X)} := e_j \cdot X \mid 1 \le j \le n,\ X \in \mathcal{I}_D \}.
\]

For each \( g \in G \), we write:
\[
\rho(g)(e_k) = \sum_i e_i \cdot a_{ik}(g), \quad
a_{ik}(g) = \sum_{X \in \mathcal{I}_D} \rho_{ik}^{(X)}(g) \cdot X,
\]
where \( \rho_{ik}^{(X)}(g) \in \mathbb{R} \) are the real matrix coefficients. Then the action on the real basis \( \mathcal{B} \) becomes:
\[
\rho(g)(e_k^{(Y)}) = \sum_{i,X} \rho_{ik}^{(X)}(g)\, e_i^{(XY)}.
\]
We adopt the convention \( e_i^{(XY)} := -e_i^{(-XY)} \) when \( XY \in -\mathcal{I}_D \), ensuring that the transformation is expressed with real coefficients in \( \mathcal{B} \).

\subsection{Structure and Trace of $D$-Equivariant Operators}

Each \( \lambda \in D^{\mathrm{op}} \) acts on \( V \cong D^n \) by right multiplication \( R_\lambda(v) = v \cdot \lambda \), yielding a real-linear operator whose structure is as follows:

\begin{theorem}[Trace Formula for $D$-Linear Operators] \label{realtrace}
In the real basis \( \mathcal{B} \), the matrix of \( R_\lambda \) consists of \( d \) identical diagonal blocks of size \( n \times n \), each equal to \( \operatorname{Re}(\lambda) \cdot I_n \). In particular,
\[
\operatorname{tr}_{\mathbb{R}}(R_\lambda) = n \cdot \operatorname{Tr}_D(\lambda), \quad \text{where } \operatorname{Tr}_D(\lambda) := d \cdot \operatorname{Re}(\lambda).
\]
\end{theorem}

\begin{proof}[Sketch]
Since \( R_\lambda(e_j^{(X)}) = e_j \cdot (X \lambda) \), the action preserves each \( D \)-fiber and gives a block-diagonal matrix with \( d \) identical \( n \times n \) blocks. The trace computation follows directly.
\end{proof}

\subsection{Real Schur Orthogonality}

We define a unified $D$-valued inner product for functions \( f_1, f_2 : G \to D \) as
\[
\langle f_1, f_2 \rangle_G := \frac{1}{|G|} \sum_{g \in G} f_1(g) \cdot \overline{f_2(g)}.
\]
When \( f_1, f_2 \in \mathbb{R} \subset D \), this reduces to the standard real inner product in \( L^2(G) \). All coefficient functions will be interpreted using this unified notation.

Let \( (\rho, V) \) be an irreducible real representation with \( \mathrm{End}_G(V) \cong D^{\mathrm{op}} \). Let \( \{e_j\} \) be a $D$-orthonormal basis and \( \mathcal{I}_D \) a real basis of \( D \). Define real-valued matrix coefficients \( \rho_{ij}^{(X)} : G \to \mathbb{R} \).

\begin{theorem}[Real Schur Orthogonality]
The functions \( \rho_{ij}^{(X)} \) form an orthogonal set in \( L^2(G) \), with
\[
\left\langle \rho_{ij}^{(X)}, \rho_{k\ell}^{(Y)} \right\rangle_G = \frac{1}{nd} \delta_{ik} \delta_{j\ell} \delta_{XY}.
\]
\end{theorem}

\begin{proof}[Sketch]
Define an operator \( T := E_{j\ell}^{(X,Y)} \) on \( \mathcal{B} \), average it over \( G \) to obtain a \( G \)-invariant operator \( \widetilde{T} \in \mathrm{End}_G(V) \cong D^{\mathrm{op}} \), and compare traces to recover the inner product.
\end{proof}

\subsection{Equivalence with \texorpdfstring{$D$}{D}-Valued Orthogonality}

We now relate the real-valued matrix coefficients \( \rho_{ij}^{(X)} \colon G \to \mathbb{R} \) introduced earlier to the corresponding \( D \)-valued coefficients \( \rho_{ij} \colon G \to D \), enabling us to reformulate orthogonality relations directly in terms of \( D \)-valued inner products.

Recall that for each \( g \in G \), we have
\[
\rho_{ij}(g) = \sum_{X \in \mathcal{I}_D} \rho_{ij}^{(X)}(g)\, X.
\]
To isolate each real component algebraically, we introduce a family of \emph{partial conjugation} maps \( \tau_X \colon D \to D \), indexed by \( X \in \mathcal{I}_D \), which we now describe.

To algebraically recover the real components of \( D \)-valued functions, we introduce a family of \emph{partial conjugation maps} \( \tau_X : D \to D \) indexed by \( X \in \mathcal{I}_D \), defined as follows:
\begin{itemize}
    \item For \( D = \mathbb{R} \), define \( \tau_1 = \mathrm{id} \).
    
    \item For \( D = \mathbb{C} \), set
    \[
    \tau_1(z) = z, \quad \tau_{\mathbf{i}}(a + b\mathbf{i}) = a - b\mathbf{i}.
    \]
    That is, \( \tau_{\mathbf{i}} \) is complex conjugation.

    \item For \( D = \mathbb{H} \), define \( \tau_X(v) := X v X^{-1} \) for each imaginary unit \( X \in \mathcal{I}_D \setminus \{1\} \). This map acts as a 180$^\circ$ rotation about the axis \( X \), fixing \( \mathrm{Span}_\mathbb{R}\{1, X\} \) and flipping the orthogonal directions.
\end{itemize}

Each \( \tau_X \) is an \( \mathbb{R} \)-algebra automorphism of \( D \), and the collection \( \{ \tau_X : X \in \mathcal{I}_D \} \) encodes natural symmetries of the algebra. In what follows, we refer to the map \( \tau_X \) applied to a function \( f : G \to D \) as its \emph{partial conjugate}, and we write:
\[
f^{[X]}(g) := \tau_X(f(g)).
\]

This notation allows us to express each real-valued coefficient \( f^{(X)} \), corresponding to expansion along the basis \( \mathcal{I}_D \), in terms of conjugated functions \( f^{[Y]} \). The next lemma establishes the algebraic link between these components.

\begin{lemma}[Key Identity for \texorpdfstring{$\tau_X$}{tau}-Conjugation]
\label{lemma:tau-conj-identity}
Let \( D \in \{ \mathbb{R}, \mathbb{C}, \mathbb{H} \} \) with standard real basis \( \mathcal{I}_D \), and let \( d = \dim_{\mathbb{R}}(D) \). Then:
\begin{enumerate}
    \item \textbf{Orthogonality sum identity:}
    \[
    \sum_{X \in \mathcal{I}_D} X \cdot \overline{ \tau_Y(X) } = d \cdot \delta_{1Y}.
    \]

    \item \textbf{Inversion formula:} For any function \( f : G \to D \), the real-valued coefficient function \( f^{(X)} : G \to \mathbb{R} \) satisfies
    \[
    f^{(X)} = \frac{1}{d} \sum_{Y \in \mathcal{I}_D} \left( \frac{\tau_X(Y)}{Y} \right) \cdot \overline{X} \cdot f^{[Y]},
    \]
    where \( f^{[Y]} := \tau_Y(f) \) denotes the partial conjugate of \( f \).
\end{enumerate}
\end{lemma}

\noindent
Note that since \( \tau_X(Y) \in \{ \pm Y \} \) for all \( X, Y \in \mathcal{I}_D \), the ratio \( \tau_X(Y)/Y \) always lies in \( \{ \pm 1 \} \). These identities show that the maps \( \tau_X \) serve as both projection and symmetry operators on \( D \). They are essential for translating real-valued orthogonality conditions for matrix coefficients into equivalent \( D \)-valued inner product identities, enabling a uniform treatment of all Frobenius--Schur types.

\begin{theorem}[Equivalence of Real and \texorpdfstring{$D$}{D}-Valued Orthogonality]
\label{d-orthogonality}
Let \( \{f_\alpha : G \to D\}_{\alpha \in I} \) be a collection of \( D \)-valued functions. For each \( X \in \mathcal{I}_D \), write
\[
f_\alpha(g) = \sum_{X \in \mathcal{I}_D} f_\alpha^{(X)}(g) X, \quad \text{with } f_\alpha^{(X)}(g) \in \mathbb{R}.
\]
Then the following are equivalent:
\begin{enumerate}
    \item The set \( \{ f_\alpha^{(X)} \} \) is orthonormal in \( L^2(G) \), i.e.,
    \[
    \left\langle f_\alpha^{(X)}, f_\beta^{(Y)} \right\rangle_G = \frac{1}{d} \delta_{\alpha\beta} \delta_{XY}.
    \]

    \item The conjugated \( D \)-valued functions satisfy
    \[
    \left\langle f_\alpha, f_\beta^{[Z]} \right\rangle_G  = \delta_{\alpha\beta} \delta_{1Z}.
    \]
\end{enumerate}
\end{theorem}

\begin{proof}
We begin by expanding \( f_\beta^{[Z]}(g) \):
\[
f_\beta^{[Z]}(g) = \tau_Z\left( \sum_{Y \in \mathcal{I}_D} f_\beta^{(Y)}(g) Y \right)
= \sum_{Y \in \mathcal{I}_D} f_\beta^{(Y)}(g) \cdot \tau_Z(Y),
\]
so
\[
\overline{f_\beta^{[Z]}(g)} = \sum_{Y \in \mathcal{I}_D} f_\beta^{(Y)}(g) \cdot \overline{ \tau_Z(Y) }.
\]

The inner product becomes
\begin{align*}
\left\langle f_\alpha, f_\beta^{[Z]} \right\rangle_G 
&= \frac{1}{|G|} \sum_{g \in G} 
\left( \sum_{X \in \mathcal{I}_D} f_\alpha^{(X)}(g) X \right)
\left( \sum_{Y \in \mathcal{I}_D} f_\beta^{(Y)}(g) \cdot \overline{ \tau_Z(Y) } \right) \\
&= \sum_{X,Y \in \mathcal{I}_D} \left\langle f_\alpha^{(X)}, f_\beta^{(Y)} \right\rangle_G \cdot X \cdot \overline{ \tau_Z(Y) }.
\end{align*}

\textbf{(1) \(\Rightarrow\) (2).} By the assumption,
\[
\left\langle f_\alpha^{(X)}, f_\beta^{(Y)} \right\rangle_G = \frac{1}{d} \delta_{\alpha\beta} \delta_{XY},
\]
and hence
\[
\left\langle f_\alpha, f_\beta^{[Z]} \right\rangle_G 
= \delta_{\alpha \beta} \cdot \frac{1}{d} \sum_{X \in \mathcal{I}_D} X \cdot \overline{ \tau_Z(X) } 
= \delta_{\alpha \beta} \cdot \delta_{1Z},
\]
by Lemma~\ref{lemma:tau-conj-identity}(1).

\medskip

\textbf{(2) \(\Rightarrow\) (1).}  
Using the inversion formula in Lemma~\ref{lemma:tau-conj-identity}(2), we write:
\[
f_\beta^{(Y)} = \frac{1}{d} \sum_{Z \in \mathcal{I}_D} \left( \frac{\tau_Y(Z)}{Z} \right) \cdot \overline{Y} \cdot f_\beta^{[Z]}.
\]

Taking the inner product gives:
\[
\left\langle f_\alpha, f_\beta^{(Y)} \right\rangle_G = \frac{1}{d} \sum_{Z \in \mathcal{I}_D} \left( \frac{\tau_Y(Z)}{Z} \right) Y \cdot \left\langle f_\alpha, f_\beta^{[Z]} \right\rangle_G.
\]

By assumption (2), the only nonzero term is when \( Z = 1 \), giving:
\[
\left\langle f_\alpha, f_\beta^{(Y)} \right\rangle_G = \frac{1}{d} \delta_{\alpha\beta} \cdot Y.
\]

On the other hand, we also have:
\[
\left\langle f_\alpha, f_\beta^{(Y)} \right\rangle_G 
= \sum_{X \in \mathcal{I}_D} \left\langle f_\alpha^{(X)}, f_\beta^{(Y)} \right\rangle_G \cdot X.
\]

Comparing both expressions, we conclude:
\[
\left\langle f_\alpha^{(X)}, f_\beta^{(Y)} \right\rangle_G = \frac{1}{d} \delta_{\alpha\beta} \delta_{XY},
\]
as claimed.
\end{proof}

\section{Group Orbits and Moment Conditions}

We develop moment-based criteria for when a group orbit \( Gv \subset S(V) \) forms a spherical design. The 1-design and 2-design conditions correspond to the vanishing of first and second moments, respectively, and can be verified using Schur orthogonality.

Our contributions are twofold: we provide an explicit characterization of 2-designs in isotypic components across all Frobenius–Schur types, and we establish a natural bijection between global 2-designs and tuples of componentwise 2-designs via dimension-weighted embeddings. While based on classical principles, this explicit formulation appears to be absent from the existing literature.

\subsection{Spherical Designs from Group Orbits}

Let \( V \) be a finite-dimensional real inner product space equipped with an orthogonal action of a finite group \( G \). For any unit vector \( v \in V \), consider its orbit
\[
Gv := \{ gv \mid g \in G \} \subset S(V),
\]
where \( S(V) \) denotes the unit sphere in \( V \).

A finite subset \( X \subset S(V) \) is called a \emph{spherical \( t \)-design} if it averages all real-valued polynomial functions of degree at most \( t \) exactly:
\[
\frac{1}{|X|} \sum_{x \in X} f(x) = \int_{S(V)} f(x)\, d\mu(x) \quad \text{for all } f \in P_{\le t}(V),
\]
where \( P_{\le t}(V) \) denotes the space of real polynomials of degree at most \( t \), and \( d\mu \) is the normalized rotation-invariant measure on the sphere.

\medskip

When \( X = Gv \) is a group orbit, the design condition can be tested by comparing orbit averages to known spherical integrals. For a standard orthonormal basis \( \{ e_1, \dots, e_n \} \) of \( V \cong \mathbb{R}^n \), the relevant integrals on \( S^{n-1} \) are:

\[
\int_{S^{n-1}} x_i\, d\mu(x) = 0, \qquad
\int_{S^{n-1}} x_i^2\, d\mu(x) = \frac{1}{n}, \qquad
\int_{S^{n-1}} x_i x_j\, d\mu(x) = 0 \quad (i \ne j),
\]
which collectively yield:
\[
\int_{S^{n-1}} xx^\top\, d\mu(x) = \frac{1}{n} I_n.
\]

\medskip

These integrals motivate the following characterization for group orbit designs.

\begin{definition}[Moment Conditions for Spherical Designs]
Let \( v \in V \) be a unit vector. The orbit \( Gv \subset S(V) \) is:
\begin{itemize}
    \item a \textbf{1-design} if and only if
    \[
    \sum_{g \in G} gv = 0;
    \]
    
    \item a \textbf{2-design} if and only if it is a 1-design and
    \[
    \frac{1}{|G|} \sum_{g \in G} (gv)(gv)^\top = \frac{1}{\dim V} I_V.
    \]
\end{itemize}
\end{definition}

The matrix condition in the second part provides a concrete criterion for identifying 2-designs via second-moment matrices.

The 1-design condition has a standard interpretation in representation theory:

\begin{proposition}
    \label{them:1-design}
Let \( V \) be a real representation of a finite group \( G \). Then the orbit \( Gv \) forms a spherical 1-design for all unit vectors \( v \in V \) if and only if \( V \) contains no copy of the trivial representation.
\end{proposition}

\subsection{Constructing 2-Designs via Orthogonal Decomposition}

We now begin our study of spherical 2-designs arising from group orbits. Unlike 1-designs, which are characterized by the absence of invariant vectors, 2-designs require a more rigid second-moment condition reflecting global isotropy. To analyze this condition systematically, we first consider the case where the representation space decomposes orthogonally into \( G \)-invariant subspaces with no shared irreducible components.

In this setting, we establish a simple and natural bijection between global 2-designs and tuples of 2-designs in the individual components, combined via dimension-weighted scaling. While the construction follows readily from moment computations and Schur orthogonality, we are not aware of this correspondence being explicitly stated in the literature.

\begin{proposition}[Combination of Disjoint Components]
Let \( V = V_1 \oplus V_2 \) be a decomposition into orthogonal \( G \)-invariant subspaces with no irreducible components in common. Suppose \( x_1 \in V_1 \), \( x_2 \in V_2 \) are unit vectors such that the orbits \( Gx_1 \subset V_1 \) and \( Gx_2 \subset V_2 \) are spherical 2-designs.

Define the dimension-scaled sum
\[
x = a x_1 \oplus b x_2, \quad
a = \sqrt{\frac{\dim V_1}{\dim V}}, \quad
b = \sqrt{\frac{\dim V_2}{\dim V}},
\]
with \( \dim V = \dim V_1 + \dim V_2 \). Then \( \|x\| = 1 \), and the orbit \( Gx \subset V \) is a spherical 2-design.
\end{proposition}
\begin{proof}
We view vectors in \( V = V_1 \oplus V_2 \) as column vectors with block components in \( V_1 \) and \( V_2 \). For the element
\[
x = a x_1 \oplus b x_2,
\]
its orbit elements have the form
\[
g x = a g x_1 \oplus b g x_2,
\]
so the outer product \( (g x)(g x)^\top \) becomes the block matrix
\[
(g x)(g x)^\top = 
\begin{bmatrix}
a^2 (g x_1)(g x_1)^\top & ab (g x_1)(g x_2)^\top \\
ab (g x_2)(g x_1)^\top & b^2 (g x_2)(g x_2)^\top
\end{bmatrix}.
\]

Averaging over \( G \), we obtain the second-moment matrix:
\[
\frac{1}{|G|} \sum_{g \in G} (g x)(g x)^\top =
\begin{bmatrix}
a^2 M_{11} & ab M_{12} \\
ab M_{21} & b^2 M_{22}
\end{bmatrix},
\]
where
\[
M_{11} = \frac{1}{|G|} \sum_{g \in G} (g x_1)(g x_1)^\top, \quad
M_{22} = \frac{1}{|G|} \sum_{g \in G} (g x_2)(g x_2)^\top,
\]
\[
M_{12} = \frac{1}{|G|} \sum_{g \in G} (g x_1)(g x_2)^\top, \quad
M_{21} = M_{12}^\top.
\]

Since \( x_1 \) and \( x_2 \) lie in orthogonal subspaces with no common irreducible summands, Schur orthogonality implies that the matrix coefficients of \( x_1 \) and \( x_2 \) are orthogonal. Therefore, the mixed terms vanish:
\[
M_{12} = M_{21}^\top = 0.
\]

Moreover, because \( Gx_1 \) and \( Gx_2 \) are spherical 2-designs in their respective spaces, we have
\[
M_{11} = \frac{1}{d_1} I_{V_1}, \quad
M_{22} = \frac{1}{d_2} I_{V_2}, \quad \text{where } d_i = \dim V_i.
\]
Substituting in the block matrix, we obtain:
\[
\frac{1}{|G|} \sum_{g \in G} (g x)(g x)^\top =
\begin{bmatrix}
\frac{a^2}{d_1} I_{V_1} & 0 \\
0 & \frac{b^2}{d_2} I_{V_2}
\end{bmatrix}.
\]

Using the definitions \( a^2 = \frac{d_1}{d_1 + d_2} \) and \( b^2 = \frac{d_2}{d_1 + d_2} \), we simplify:
\[
\frac{a^2}{d_1} = \frac{1}{d_1 + d_2}, \quad
\frac{b^2}{d_2} = \frac{1}{d_1 + d_2}.
\]
Thus,
\[
\frac{1}{|G|} \sum_{g \in G} (g x)(g x)^\top = \frac{1}{\dim V}
\begin{bmatrix}
I_{V_1} & 0 \\
0 & I_{V_2}
\end{bmatrix}
= \frac{1}{\dim V} I_V,
\]
which confirms that \( Gx \subset S(V) \) is a spherical 2-design.
\end{proof}

\begin{remark}
The dimension-weighted scaling ensures that contributions from each component are balanced so that the resulting second-moment matrix is scalar on the whole space.
\end{remark}

We now consider the inverse direction: given a spherical 2-design in \( V \), how does it project to components? We show that orthogonal projections onto \( G \)-invariant subspaces preserve the 2-design property, and that the original design can be reconstructed from its components via proper rescaling.

\begin{proposition}[Projection Preserves 2-Designs]
Let \( V \) be a real \( G \)-representation of dimension \( n \), and \( W \subset V \) a \( G \)-invariant subspace of dimension \( m \). If \( Gv \subset S(V) \) is a spherical 2-design and \( w := \pi_W(v) \ne 0 \), then the normalized projection orbit
\[
\left\{ \frac{g w}{\|w\|} \,\bigg{|}\, g \in G \right\} \subset S(W)
\]
is a spherical 2-design in \( W \). Moreover, \( \|w\|^2 = \frac{m}{n} \).
\end{proposition}

\begin{proof}
Let \( h : W \to \mathbb{R} \) be a homogeneous harmonic polynomial of degree \( j \le 2 \), and define \( f(x) := h(\pi_W(x)) \), a harmonic polynomial on \( V \). Since \( \pi_W \) is linear and \( h \) is harmonic, \( f \) is harmonic on \( V \).

By the 2-design property of \( Gv \subset S(V) \),
\[
\sum_{g \in G} f(gv) = \sum_{g \in G} h(\pi_W(gv)) = \sum_{g \in G} h(gw) = 0,
\]
where we used \( \pi_W(gv) = g \pi_W(v) = gw \) due to orthogonality and \( G \)-invariance of \( W \).

Homogeneity of \( h \) yields:
\[
\sum_{g \in G} h\left( \frac{gw}{\|w\|} \right) = \frac{1}{\|w\|^j} \sum_{g \in G} h(gw) = 0.
\]
Hence, the normalized orbit \( G(w/\|w\|) \) is a spherical 2-design in \( W \).

To compute \( \|w\|^2 \), consider a coordinate function \( x_1 \) on \( W \subset V \), which extends naturally to \( V \). Let \( f(x) := x_1^2 \). Since \( \pi_W(v) = w \), we have \( f(\pi_W(v)) = f(v) \), and hence by the 2-design property:
\[
\sum_{g \in G} f(gw) = \sum_{g \in G} f(gv)  = \int_{S(V)} x_1^2 \, d\mu = \frac{1}{n}.
\]
On the other hand, homogeneity gives
\[
\sum_{g \in G} f\left( \frac{g w}{\|w\|} \right) = \frac{1}{\|w\|^2} \sum_{g \in G} f(gw) = \int_{S(W)} x_1^2 \, d\mu = \frac{1}{m}.
\]
Solving yields \( \|w\|^2 = \frac{m}{n} \). \qedhere
\end{proof}

The preceding results establish both directions of correspondence: a global 2-design in $V$ can be assembled from 2-designs in orthogonal $G$-invariant components, and conversely, each global design decomposes into such componentwise designs via projection and normalization. We now formalize this relationship as a bijection. While the construction holds for any orthogonal decomposition with no overlapping irreducible components, it is most naturally applied when $V$ is decomposed into its isotypic components, as considered in the next subsection.

\begin{theorem}[Bijection via Orthogonal Decomposition]\label{thm:orthogonaldecomposition}
Let $V = \bigoplus_{i=1}^r V_i$ be an orthogonal decomposition of a real $G$-representation into $G$-invariant subspaces with no shared irreducible subrepresentations. Then there is a bijection between:

\begin{itemize}
\item spherical 2-designs $Gx \subset S(V)$, and
\item tuples $(x_1, \dots, x_r)$, where each $x_i \in V_i$ is a unit vector such that $Gx_i \subset S(V_i)$ is a spherical 2-design,
\end{itemize}

via the dimension-weighted embedding

$$
x = \sum_{i=1}^r \sqrt{\frac{\dim V_i}{\dim V}}\, x_i.
$$

Conversely, each such $x \in S(V)$ determines the corresponding tuple $(x_1, \dots, x_r)$ by projection and normalization:

$$
x_i = \frac{\pi_{V_i}(x)}{\|\pi_{V_i}(x)\|}, \quad \text{with} \quad \|\pi_{V_i}(x)\|^2 = \frac{\dim V_i}{\dim V}.
$$

\end{theorem}

\subsection{Classification of 2-Designs in a Single Isotypic Component}

We now turn to the case where the group orbit lies entirely within a single isotypic component of a real representation. This setting offers a particularly structured framework in which the second-moment condition can be analyzed precisely using the Frobenius–Schur type of the underlying irreducible representation.

Let \( (\rho, V) \) be a non-trivial irreducible real representation of \( G \), and let \( D = \mathrm{End}_G(V) \in \{ \mathbb{R}, \mathbb{C}, \mathbb{H} \} \) denote its Frobenius–Schur type. Then the corresponding isotypic component of \( \mathbb{R}[G] \) is isomorphic, as a real \( G \)-representation, to the matrix space \( M_n(D) := \mathrm{Mat}_{n \times n}(D) \), where \( G \) acts on the left via a \( D \)-linear representation \( \rho : G \to \mathrm{GL}_n(D) \).

Within \( M_n(D) \), we consider the subspaces \( \mathrm{Mat}_{n \times m}(D) \) consisting of the first \( m \) columns. These are real \( G \)-subrepresentations, isomorphic to \( V^m \), and serve as a natural setting for constructing and analyzing matrix-valued orbits. Let \( \pi_{n,m}^D: M_n(D) \to \mathrm{Mat}_{n \times m}(D) \) denote the orthogonal projection onto these subspaces.

We now state a complete characterization of when such orbits form spherical 2-designs. This is governed by a simple algebraic condition on the matrix's inner product structure:

\begin{theorem}[2-Designs in Isotypic Components]\label{thm:isotypic-2design}
Let \( M \in \mathrm{Mat}_{n \times m}(D) \) have Frobenius norm \( \|M\|_F = 1 \), and let \( G \) act on \( M \) by left multiplication via \( \rho \). Then the orbit
\[
GM := \{ \rho(g) M \mid g \in G \}
\]
forms a spherical 2-design in \( \mathbb{R}^{nmd} \) if and only if
\[
\overline{M}^\top M = \frac{1}{m} I_m.
\]

In particular, for \( m = n \), this holds if and only if \( M = \frac{1}{\sqrt{n}} U \) for some unitary matrix \( U \in \mathrm{O}_D(n) \).

Moreover, every spherical 2-design orbit in a real subrepresentation of \( M_n(D) \) arises as the image
\[
\pi_{n,m}^D\left( \left\{ \frac{1}{\sqrt{m}} \rho(g) U \mid g \in G \right\} \right), \quad \text{for some } U \in \mathrm{O}_D(n).
\]
\end{theorem}

\begin{proof}
By Proposition~\ref{them:1-design}, every group orbit is a 1-design. We now verify the 2-design condition.

Write \( M \) in real coordinates as
\[
M = \sum_{X \in \mathcal{I}_D} M^{(X)} \cdot X, \quad M^{(X)} \in \mathrm{Mat}_{n \times m}(\mathbb{R}).
\]
According to Theorem~\ref{d-orthogonality}, the orbit \( GM \subset \mathbb{R}^{nmd} \) forms a 2-design if and only if
\[
\left\langle (\rho(g)M)_{ki},\, (\rho(g)M)_{\ell j}^{[Y]} \right\rangle_G = \frac{1}{nm} \delta_{k\ell} \delta_{ij} \delta_{1Y}
\]
for all \( i,j,k,\ell \) and \( Y \in \mathcal{I}_D \), where \( \tau_Y \) denotes the partial conjugation associated to \( Y \).

Expanding both entries via the real components of \( \rho(g) \) and \( M \), we get:
\[
(\rho(g)M)_{ki} = \sum_{s,X,W} \rho(g)_{ks}^{(X)} M_{si}^{(W)} (XW),
\]
\[
\overline{\tau_Y((\rho(g)M)_{\ell j})} = \sum_{t,X',Z} \rho(g)_{\ell t}^{(X')} M_{tj}^{(Z)} \cdot \overline{\tau_Y(X'Z)}.
\]

Using the Schur orthogonality for real matrix coefficients:
\[
\frac{1}{|G|} \sum_{g \in G} \rho(g)_{ks}^{(X)} \rho(g)_{\ell t}^{(X')} = \frac{1}{nd} \delta_{k\ell} \delta_{st} \delta_{XX'},
\]
we obtain:
\begin{equation} \label{eq:general-inner-product}
\left\langle (\rho(g)M)_{ki}, (\rho(g)M)_{\ell j}^{[Y]} \right\rangle_G
= \frac{1}{nd} \delta_{k\ell} \sum_s \sum_{W,Z} M_{si}^{(W)} M_{sj}^{(Z)} \sum_{X \in \mathcal{I}_D} (XW) \cdot \overline{\tau_Y(XZ)}.
\end{equation}

We now evaluate this expression for each \( D \in \{ \mathbb{R}, \mathbb{C}, \mathbb{H} \} \).

\medskip

\noindent \textbf{Case \( D = \mathbb{R} \):}
Here \( \mathcal{I}_D = \{1\} \) and \( \tau_1 = \mathrm{id} \), so the inner sum becomes 1. Then
\[
\left\langle (\rho(g)M)_{ki}, (\rho(g)M)_{\ell j} \right\rangle_G = \frac{1}{n} \delta_{k\ell} \sum_s M_{si} M_{sj}.
\]
Thus, \( GM \) is a 2-design if and only if
\[
M^\top M = \frac{1}{m} I_n.
\]

\medskip

\noindent \textbf{Case \( D = \mathbb{C} \):}
Let \( \mathcal{I}_D = \{1, \mathbf{i}\} \). Since the elements commute and \( \tau_Y \) is complex conjugation when \( Y = \mathbf{i} \), we have
\[
\sum_{X \in \mathcal{I}_D} (XW) \cdot \overline{\tau_Y(XZ)} = \left( \sum_X X \overline{\tau_Y(X)} \right) W \overline{\tau_Y(Z)} = 2 \delta_{1Y} W \bar{Z}.
\]
Substituting into \eqref{eq:general-inner-product}:
\[
\left\langle (\rho(g)M)_{ki}, (\rho(g)M)_{\ell j}^{[Y]} \right\rangle_G
= \frac{1}{n} \delta_{k\ell} \delta_{1Y} \sum_s M_{si} \overline{M_{sj}}.
\]
Hence, the 2-design condition is equivalent to
\[
\overline{M}^\top M = \frac{1}{m} I_n.
\]

\medskip

\noindent \textbf{Case \( D = \mathbb{H} \):}  
Let \( \mathcal{I}_D = \{1, \mathbf{i}, \mathbf{j}, \mathbf{k}\} \) and recall that \( \tau_Y(x) = Y x Y^{-1} \) for all \( x \in \mathbb{H} \). Then:
\[
\sum_{X \in \mathcal{I}_D} (XW) \cdot \overline{ \tau_Y(XZ) }
= \sum_{X} (XW) \cdot \overline{ Y X Z Y^{-1} }
= \sum_{X} X (W Y \bar{Z}) \bar{X} \cdot \bar{Y}.
\]

We use the basic identity
\[
\sum_{X \in \mathcal{I}_D} X T \bar{X} = d \cdot \delta_{1T} \quad \text{for all } T \in \mathcal{I}_D,
\]
which holds by symmetry and is straightforward to verify.
.

Applying this to the expression above, the sum is nonzero if and only if:
\[
W Y \bar{Z} = \pm 1 \quad \Leftrightarrow \quad Y = \pm \bar{W} Z \quad \text{or equivalently} \quad \bar{Y} = \pm \bar{Z} W.
\]

Therefore,
\[
\sum_{X \in \mathcal{I}_D} (XW) \cdot \overline{ \tau_Y(XZ) } = \left(\sum_{X} X (W Y \bar{Z}) \bar{X} \right) \bar{Y} = d \cdot \delta_{\bar{W} Z, \pm Y} \cdot \bar{Z} W.
\]
Note that the sign ambiguity in \( Y = \pm \bar{W} Z \) does not affect the result, since the same sign appears in \( \bar{Y} \), and the two signs always cancel in the product.

Substituting into \eqref{eq:general-inner-product} gives:
\begin{equation} \label{eq:quaternion-inner}
\left\langle (\rho(g)M)_{ki}, (\rho(g)M)_{\ell j}^{[Y]} \right\rangle_G = \frac{1}{n} \delta_{k\ell} \sum_{s, W,Z} \delta_{\bar{W}Z,\pm Y} M_{si}^{(W)} M_{sj}^{(Z)} \bar{Z} W.
\end{equation}

Summing both sides over \( Y \in \mathcal{I}_D \), we obtain:
\begin{equation} \label{eq:quaternion-trace}
\sum_Y \left\langle (\rho(g)M)_{ki}, (\rho(g)M)_{\ell j}^{[Y]} \right\rangle_G = \frac{1}{n} \delta_{k\ell} \sum_s \overline{M_{si}} M_{sj}.
\end{equation}

Now suppose \( GM \) is a 2-design. Then the left-hand side equals \( \frac{1}{nm} \delta_{k\ell} \delta_{ij} \). Therefore,
\[
\overline{M}^\top M = \frac{1}{m} I_n.
\]

Conversely, assume this identity holds. Then 
$$ (\overline{M}^\top M)_{i,j} = \sum_{s, W,Z} M_{sj}^{(Z)} M_{si}^{(W)} \bar{Z} W = \frac{1}{m}\delta_{ij} $$
comparing the \( Y \)-components of both sides of \eqref{eq:quaternion-inner}, we conclude
\[
\left\langle (\rho(g)M)_{ki}, (\rho(g)M)_{\ell j}^{[Y]} \right\rangle_G = \frac{1}{nm} \delta_{k\ell} \delta_{ij} \delta_{1Y},
\]
and hence \( GM \) is a spherical 2-design.

\medskip

For the final claim, note that if \( \overline{M}^\top M = \frac{1}{m} I_m \), then the columns of \( M \) are orthonormal (up to scale), so \( M = \pi_{n,m}^D(U)/\sqrt{m} \) for some unitary matrix \( U \in \mathrm{O}_D(n) \). Therefore, \( GM \) arises as the image of the scaled orbit \( \{ \rho(g) U/\sqrt{m} \} \) under \( \pi_{n,m}^D \).

\end{proof}

\subsection{Classification of Spherical 2-Design Orbits}

We now combine the results of Subsections~3.3 and~3.4 to obtain a complete classification of spherical 2-designs arising from group orbits in real representations. The key observation is that any real subrepresentation \( V \subset \mathbb{R}[G]_0 \) admits an orthogonal decomposition into isotypic components
\[
V = \bigoplus_i V_i,
\]
where each \( V_i \) is isomorphic to a matrix space \( \mathrm{Mat}_{n_i \times m_i}(D_i) \) for some \( D_i \in \{\mathbb{R}, \mathbb{C}, \mathbb{H}\} \), corresponding to the Frobenius–Schur type of the underlying irreducible representation.

Within each component \( V_i \), the 2-design condition is characterized by the moment identity
\[
\overline{M}_i^\top M_i = \frac{1}{m_i} I_{m_i}
\]
for a representative matrix \( M_i \in \mathrm{Mat}_{n_i \times m_i}(D_i) \). As shown in Section~3.3, such designs arise as projections of scaled orbits in the unitary group \( \mathrm{O}_{D_i}(n_i) \). Section~3.2 then guarantees that dimension-weighted combinations of such componentwise 2-designs yield a global 2-design in \( V \), and that every global design arises in this way.

Thus, Theorem \ref{thm:main2} stated in the introduction follows immediately by applying the bijection from Theorem \ref{thm:orthogonaldecomposition} to the isotypic decomposition of \( V \), with each component governed by the criterion in Theorem \ref{thm:isotypic-2design}.

\bibliographystyle{plain}
\bibliography{reference}

\end{document}